\documentclass[a4paper,10pt]{article}
\usepackage{amscd}
\usepackage{amsmath}
\usepackage{bbding}
\usepackage{amsfonts}
\usepackage{cite}
\usepackage{mathrsfs}
\usepackage{amssymb}
\allowdisplaybreaks[2]
\usepackage{leftidx}
\usepackage{graphicx,latexsym,amsmath,amssymb,amsthm,enumerate}
\usepackage{float}

\usepackage[dvipdfm,
           pdfstartview=FitH,
           colorlinks=true,
             pdfborder=001,
           citecolor=blue]{hyperref}

9.3in\textwidth 6.4in \hoffset -2.0cm \DeclareFontEncoding{OT2}{}{}
\DeclareTextSymbol{\cyrsftsn}{OT2}{126}
\DeclareTextSymbol{\textnumero}{OT2}{125}

\theoremstyle{definition}
\newtheorem{theorem}{Theorem}[section]
\newtheorem{lemma}{Lemma}[section]

\newtheorem{definition}{Definition}[section]
\newtheorem{remark}{Remark}[section]
\newtheorem{example}{Example}[section]

\begin{document}
\title{{\LARGE\bf{Global Mittag-Leffler stability of Fractional-Order Projection Neural Networks with Impulses}}\thanks{This work was supported by the National Natural Science Foundation of China (11901273, 11471230, 11671282) and Application Fundamentals Foundation of Science and Technology Department of Sichuan (No. 2020YJ0366) and the Key Project in Universities of Henan Province (No. 19A110025).}}
\author{ Jin-dong Li$^{a,b}$, Zeng-bao Wu$^{c}$, Nan-jing Huang$^b$\thanks{Corresponding author. E-mail addresses: nanjinghuang@hotmail.com; njhuang@scu.edu.cn} \\
{\small\it $^a$College of Management Science, Chengdu University of Technology, Chengdu, Sichuan 610059, P.R. China}\\
{\small\it $^b$Department of Mathematics, Sichuan University, Chengdu, Sichuan 610064, P.R. China}\\
{\small\it $^c$Department of Mathematics, Luoyang Normal University, Luoyang, Henan  471934, P.R. China}}
\date{ }
\date{ }
\maketitle
\begin{flushleft}
\hrulefill\\
\end{flushleft}
 {\bf Abstract}.
This paper is about the study of a new class of  fractional-order projection neural networks with impulses which capture the desired features of both the variational inequality and the fractional-order impulsive dynamical systems within the same framework. We obtain  the existence and boundedness of solutions for such fractional-order projection neural networks under mild conditions. Moreover, we give some sufficient conditions for ensuring the global Mittag-Leffler stability of the equilibrium point for such fractional-order projection neural networks by utilizing a general quadratic Lyapunov function. Finally, we provide two numerical examples to illustrate the validity and feasibility of the main results.
\begin{flushleft}
\hrulefill
\end{flushleft}

\section{Introduction} \noindent

As an important generalization of integral calculus, fractional calculus deals with the study of so-called fractional order integral and derivative operators over real or complex domains, which has been extensively used to model many practical problems arising in physics and mechanics,  biology and chemistry, economy and finance, science and engineering and so on \cite{IP,Butzer,Kilbas,Die,Sun,Zhou,Ahmed,Song,Magin}. On the other hand, neural networks have attracted wide attention owing to their applications in various fields such as robot, aerospace joint memory, pattern recognition, signal processing, and automatic control engineering. The practical application of neural networks depends on the development of related theories of neural networks, such as the study of the existence and uniqueness of solutions and the stability of neural networks.  Recently, the stability of fractional-order neural networks has drawn much attention such as Mittag-leffler stability \cite{Wu,Yang,Syed1,Meng}, asymptotic stability \cite{Liu,LB,CDLi}, and finite-time stability \cite{You,Feng,Wang,Chen}.

It is well known that, the projection neural network (dynamical system), captured the desired features of both the variational inequality and
the dynamical systems within the same framework,  can be used to solve many constrained optimization problems, variational inequality problems,  equilibrium point problems, dynamic traffic networks and so on (see, for example, \cite{Fr,CHT,WZH2,WZ,WLH,DH,LHuang,WCH,LWHuang,Ha,Li,WZH1,XV,Xia,WZHuang} and the references therein).
Taking into account of the advantages of fractional calculus, Wu and Zou \cite{WZ}, for the first time, proposed a class of fractional order projective dynamical systems. Wu et al.\cite{WLH} proposed a new system of global fractional-order interval implicit projection neural networks and showed Mittag-Leffler stability of the equilibrium point for such fractional-order projection neural networks under suitable conditions. On the basic of the linear matrix inequality technique, Li et al. \cite{LWHuang} obtained some sufficient conditions to ensure the asymptotical stability of the equilibrium point of the addressed projection neural networks. Wu et al. \cite{WZHuang} investigated a new class of global fractional-order projection dynamical system with delay and obtained existence and uniqueness of solutions for considered dynamical system  by using the Krasnoselskii fixed point theorem.

It is worth mentioning that, in the real world, many systems are often disturbed suddenly, and systems changes suddenly in a short time. This phenomenon is called as impulse. Recently, various theoretical results, numerical algorithms with applications have been studied extensively for the fractional-order neural networks with impulses under different conditions in the literature; for instance, we refer the reader to \cite{JRWang,QKSong,THYu,Chen2} and the references therein. Very recently,  Li and Kao \cite{Kao} investigated the Mittag-Leffler stability for a new coupled system of fractional-order differential equations with impulses by utilizing the direct graph theory.  Ali et al. \cite{Syed2} studied the impulsive effects on the stability equilibrium solution for Riemann-Liouville fractional-order fuzzy BAM neural networks with time delay under mild conditions. Pratap et al. \cite{JDCao} introduced a class of delayed fractional-order competitive neural networks with impulsive effects and established the stability and synchronization criteria of the considered networks.  Popa \cite{Popa} studied the global $\mu$-stability of neutral-type impulsive complex-valued BAM neural networks with leakage delay and unbounded time-varying delays.  Nevertheless, in some practical situations, it is necessary to consider the fractional-order projection neural networks with impulsive effects. However, to the best of our knowledge, the discipline of the fractional-order projection neural networks with impulsive effects is still not explored and much is desired to be done. The present work is to make an attempt in this new direction.

We now briefly sketch the contents of the paper. In the next section we present some necessary definitions, lemmas and model description. After that Section 3 characterizes the solution of fractional-order projection neural networks with impulses. In Section 4, we investigate the Mittag-leffler stability of fractional-order  projection neural networks with impulses under under some mild conditions.  Two numerical examples are given in Section 5 to confirm the validity of our main results, before we summarize the results in Section 6.

\section{Preliminaries}\noindent
\setcounter{equation}{0}
Let $R=(-\infty, +\infty)$, $Z^{+}=\{1,2,\cdots\}$, and $C$ be the set of all complex numbers. Let $R^n$ be the $n$-dimensional Euclidean space, $R^{m\times n}$ the set of all $m\times n$ real matrices, and $I$ the identity matrix with appropriate dimension. Let $A^T$ and $A^{-1}$ denote the transpose and the inverse of matrix $A$, respectively. Assume that $\|z\|=\sqrt{z^{T}z}$ stands for the Euclidean norm of a real vector $z$. For a real matrix $A$, $\lambda_{\max}(A)$ and $\lambda_{\min}(A)$ denote, respectively, the maximal and minimal eigenvalues of $A$, and $A>0$ $(A<0)$ means the matrix $A$ is symmetric and positive definite (or negative definite). The norm of a real matrix $A$ is defined by $\|A\|=\sqrt{\lambda_{\max}(A^{T}A)}$.  In addition, let $J=[0, T]$ and $C(J,X)$ be the Banach
space of all continuous functions from $J$ into Banach space $X$ with a norm $\|u\|_C= \sup\{\|u(t)\|:t \in J\}$  for $u(t) \in C(J,X)$.  Moreover, let $PC(J,X)=\{u :J \rightarrow X :u \in C((t_k, t_{k+1}], X), k=0,1,\cdots,m, \;
\text{there exist} ~ u(t_k^-) \; \text{and}~ u(t_k^+)~\text{with} ~ u(t_k^-)=u(t_k)\}$ with norm $\|u\|_{PC}= \sup\{\|u(t)\|:t \in J\}.$

\begin{definition}\cite{IP,Kilbas} The Riemann-Liouville fractional integral of order $q$ for a function $x(t)$ is defined as
\label{def2.1}
\begin{equation}
\leftidx{_{ t_0}}D{_t^{-q}}x(t)=\frac{1}{\Gamma(q)}\int_{t_0}^t(t-s)^{q-1}x(s)ds,
\end{equation}
where $q>0$ and $t>t_0$. The Gamma function $\Gamma(q)$ is defined by the integral $$\Gamma(q)=\int_{0}^{+\infty}s^{q-1}e^{-s}ds.$$
\end{definition}

\begin{definition}\cite{IP,Kilbas} The Riemann-Liouville fractional derivative of order $q$ for a function $x(t)$ is defined as
\label{def2.2}
\begin{equation}
\leftidx{_{ t_0}}D{_t^{q}}x(t)=\frac{1}{\Gamma(n-q)}\frac{d^n}{dt^n}\int_{t_0}^t(t-s)^{n-q-1}x(s)ds,
\end{equation}
where $n-1\leq q < n$ and $n\in {Z^{+}}$.
\end{definition}

\begin{definition}\cite{IP,Kilbas} The Caputo fractional derivative of order $q$ for a function $x(t)$ is defined as
\label{def2.3}
\begin{equation}
\leftidx{_{ t_0}^C}D{_t^{q}}x(t)=\frac{1}{\Gamma(n-q-1)}\int_{t_0}^t(t-s)^{n-q-1}x^{(n)}(s)ds,\quad t>t_0,
\end{equation}
where $ n-1<q < n$ and $n\in {Z^{+}}$.
\end{definition}
Moreover, the Mittag-Leffler function with two parameters $\alpha>0$ and $\beta>0$ is defined by
$$E_{\alpha,\beta}(z)=\sum_{k=0}^\infty\frac{z^k}{\Gamma(\alpha k+\beta)},\quad \alpha>0, \; \beta>0, \; z \in C.$$
For $\beta=1$, the one-parameter Mittag-Leffler function is given by
$$E_\alpha(z):=E_{\alpha,1}(z)=\sum_{k=0}^\infty\frac{z^k}{\Gamma(\alpha k+1)}, \quad \alpha>0, \; z \in C.$$
In particular, $E_1(z)=e^z.$

\begin{lemma}\label{lemma2.1}\cite{Liu}
Let $x(t)\in R^{n}$ be a vector of differentiable function. Then, for any
time instant $t \geq t_{0}$, the following relationship holds:
$$\frac{1}{2}\leftidx{_{ t_0}^C}D{_t^{q}}\big(x^{T}(t)Qx(t)\big)\leq x^{T}(t)Q\left(\leftidx{_{ t_0}^C}D{_t^{q}}x(t)\right), \quad q \in (0,1),$$
where $Q\in  R^{n\times n}$ is a constant, square, symmetric and positive definite matrix.
\end{lemma}

\begin{lemma}\label{lemma2.2}\cite{Zhang}
For the given positive scalar $\rho >0, \vartheta, \upsilon \in R^{n}$ and matrix $Q$, then
$$\vartheta^{T}Q\upsilon\leq \frac{\varrho^{-1}}{2}\vartheta^{T}QQ^T\vartheta+\frac{\varrho}{2}\upsilon\upsilon^T.$$
\end{lemma}
\begin{definition}\cite{WZH2}
\label{def2.3}
 If $K$ is a nonempty closed convex subset of $ R^{n}$, then the projection operator $P_{K}:R^n \rightarrow K$ is defined as follows
\begin{equation*}
P_{K}(y) =\arg\min_{z\in K}\Vert y-z\Vert, \quad \forall y \in R^n.
\end{equation*}
\end{definition}

\begin{lemma}\label{lemma2.3}\cite{Kinder}
If $K$ is a convex closed subset of $R^n$, then the projection $P_K$ is non-expansive, i.e.,
\begin{equation*}
\left\|P_K(u)-P_K(v)\right\| \leq \|u-v\|, \quad \forall u,v \in R^n.
\end{equation*}
\end{lemma}

\begin{lemma}\label{lemma2.4}\cite{Ye}
For a $\alpha>0 $, suppose $a(t)$ is a nonnegative, nondecreasing function locally
integrable on $0\leq t <T $ some $(T \leq +\infty )$, and $ b(t) \leq M$ is a nonnegative, nondecreasing continuous function defined
on $0\leq t <T $, where $M$ is a constant. If $x(t)$ is nonnegative and locally integrable on $0\leq t <T $ satisfying
$$x(t)=a(t) + b(t)\int_{0}^t (t-s)^{\alpha-1}x(s)ds,$$
then $x(t)\leq a(t)E_\alpha(b(t)\Gamma(\alpha)t^\alpha).$
\end{lemma}

Let $J=[0,T]$ and  $J'=J\setminus \{t_{1},t_{2},\cdots,t_{m}\}$. This paper is addressed to study the following fractional-order projection neural networks with impulses (FPNNI):
\begin{equation} \label{eq2.3}
\left\{
\begin{array}{l}
\leftidx{_0^C}D{_t^\alpha}x_{i}(t)=-x_{i}(t)+P_{K}\left( x_{i}(t)-\rho \sum \limits_{j=1}^{n}a_{ij}x_{j}(t)-\rho b_{i}\right),\ t\in J', \\
\triangle x_{i}(t_k) =x_{i}(t_k^{+})-x_{i}(t_k^{-})=U_{ki}(x_{i}(t_k)),\\
x_{i}(0)=x_{i0}, \; i=1,2,\cdots,n; \; k=1,2,\cdots,m,
\end{array}
\right.
\end{equation}
or in the matrix-vector notation

\begin{equation} \label{eq2.4}
\left\{
\begin{array}{l}
\leftidx{_0^C}D{_t^\alpha}x(t)=-x(t)+P_{K}\left(x(t) -\rho Ax(t)-\rho b\right),\ t\in J', \\
\triangle x(t_k) =x(t_k^{+})-x(t_k^{-})=U_{k}(x(t_k)),\\
x(0)=x_0,k=1,2,\cdots,m,
\end{array}
\right.
\end{equation}
where $\leftidx{_0^C}D{_t^\alpha}x(t)$ denotes an $\alpha$ order Caputo fractional derivative of $x(t)$ with $0 < \alpha < 1$, $n$ is the number of neurons in the indicated neural networks, $x(t) = (x_1(t), x_2(t), \cdots, x_n(t))^{T} $ is the state vector of the networks at time $t$, $K$ is a closed and convex subset of $R^{n}$, $P_{K}$ is the projection operator, $A=(a_{ij})_{n\times n} \in R^{n\times n}$, $U_{k}:R^{n} \longrightarrow R^{n}$ stands for the jump operator of impulsive, and impulsive moments $\{t_k,k=1,2,3,\cdots\}$ satisfy $ 0=t_0 < t_1 < t_2 <\cdots <t_k < t_{k+1}<\cdots <t_m < t_{m+1}=T$; $x(t_k^{+})=\lim_{t\rightarrow t_k^{+}}x(t_k)$ and $x(t_k^{-})=\lim_{t\rightarrow t_k^{-}}x(t_k)$ express the right and left limits of $x(t)$ at $t=t_k$, moreover,
$x(t_k^{-})=x(t_k).$

\begin{definition}\label{def2.5}
A function $x(t):J \rightarrow R^{n}$ is said to be a solution of system (\ref{eq2.4}) on $J$ with the initial condition $x(0)=x_{0}$ if
\begin{itemize}
\item [(i)] $x(t)$ is absolutely continuous on each interval $(t_k,t_{k+1}]$ with $k=1,2,\cdots,m$;
\item[(ii)] $x(t)$ satisfies the system (\ref{eq2.4}) for almost all $t\in J$;
\item[(iii)] For any  $t_k \in (0,\infty)$, $x(t_{k}^{-})$ and $x(t_{k}^{+})$ exist such that $x(t_{k}^{-})=x(t_k)$ for $k=1,2,\cdots,m$.
\end{itemize}
\end{definition}

\begin{definition}\label{def2.6}
A constant vector  $x^{\ast} \in R^{n}$ is said to be an equilibrium point of system (\ref{eq2.4}) if $x^*$ satisfies the following equation
$$-x^*+P_{K}(x^* -\rho Ax^*-\rho b)=0$$
and $U_k(x_k^*)=0$.
\end{definition}

\begin{definition}\label{def2.7}
An equilibrium point $x^*$ of system (\ref{eq2.4}) is said to be globally Mittag-Leffler stable if there exist two constants $\lambda>0$ and $b>0$ such that
\begin{equation*}
\left\|x(t)- x^\ast\right\| \leq [m\left(x_0-x^\ast\right) E_\alpha\left(-\lambda t)^\alpha\right)]^b,
\end{equation*}
where $m(x):R^n\to R$ is local Lipschitz continuous with $m(0)=0$, $m(x)\geq 0$ for all $x\in R^n$, and $E_\alpha$ is a one-parameter Mittag-Leffler function. The system (\ref{eq2.4}) is said to be globally Mittag-Leffler stable providing its equilibrium point is globally Mittag-Leffler stable.
\end{definition}

\begin{definition}\label{def2.8}\cite{Lak}
A function $V: [0, +\infty)\times R^n \rightarrow  [0, +\infty)$ is said to belong to class $V_0$ if
\begin{itemize}
\item [(1)] $V(t,u)$ is continuous in $(t_{k-1},t_k)\times R^n $, for each $u \in R^n, \lim_{(t,v)\rightarrow (t_k^+,u)}V(t,v)=V(t_k^+,u)$ exits, and $V(t_k^-,u)=V (t_k,u)$, $k=1, 2, \cdots, m$;
\item[(2)] $ V(t,u) $ is locally Lipschitz continuous with respect to $u$.
\end{itemize}
\end{definition}

\begin{lemma}\label{lemma2.5}\cite{Yang}
Consider the following impulsive fractional-order differential equation system
\begin{equation} \label{eq2.5}
\left\{
\begin{array}{l}
\leftidx{_0^C}D{_t^\alpha}u(t)=g(t,u(t)),\ t\in J', \\
\triangle u(t_k) =u(t_k^{+})-u(t_k^{-})=\Upsilon_{k}(u(t_k)),\\
u(0)=u_0,k=1,2,\cdots,m.
\end{array}
\right.
\end{equation}
Suppose that $g(t,0)=0, t>0, \Upsilon_k(0) = 0$ and $\varphi(u ) = u + \Upsilon(u ) \geq 0$ are nondecreasing with respect to $u$, $k=1,2,\cdots,m$,  and there exists  $V(t)\in V_0$ with $V(t,0)=0$ such that
\begin{itemize}
\item  $\gamma_1\|u(t)\|^p\leq V(t,u(t)) \leq\gamma_2\|u(t)\|^{pq}, (t,u(t))\in J'\times R^n$;
\item  $\leftidx{_0^C}D{_t^\alpha}V(t,u(t))\leq-\gamma_3V(t,u(t)), (t,u(t))\in (t_{k-1},t_k)\times R^n, k=1,2,\cdots,m$;
\item $V(t^{+}, u(t)+\Upsilon_k(u(t)))\leq \gamma_4V(t,u(t)), u\in R^n, t=t_k, k=1,2,\cdots,m$.
\end{itemize}
Here $\gamma_4\in(0,1], \gamma_1,\gamma_2, \gamma_3, p$ and $q$ are arbitrary positive constants.  Then the zero solution of (\ref{eq2.5}) is globally Mittag-Leffler stable.
\end{lemma}

\begin{definition}\label{def2.7}\cite{EZ}
Let $\Lambda$ be a bounded subset in a metric space $(X,d)$. Then Kuratowskii noncompactness measure $\mu(\Lambda)$ is defined as follows:
$$\mu(\Lambda)=\inf \left\{\varepsilon>0: \Lambda \subset \displaystyle\bigcup_{i=1}^n Q_i, \, Q_i \subset E, \, \text{diam}(Q_i)<\varepsilon  (i=1,\ldots,n), n\in {Z^{+}} \right\},$$
where $\text{diam}(Q_i)=\sup\left\{d(x,y): \;  x,y \in Q_i \right\}.$
\end{definition}

\begin{definition}\label{def2.8}\cite{EZ}
Let $N:X \rightarrow X$ be a bounded and continuous operator on a Banach
space $X$.  We say that $N$ is condensing if $\mu (N(B))< \mu (B)$ for all bounded set $B\subset D(N)$, where $\mu$ is the Kuratowskii noncompactness measure.
\end{definition}

\begin{lemma}\label{lemma2.6}\cite{MA}
Let $F=F_{1}+F_{2}:X \rightarrow X $. Assume that
\begin{itemize}
\item [(i)] $F_{1}$ is $k$-contraction with $0\leq k<1$, that is $\|F_{1}(x)-F_{1}(y)\|\leq k \|x-y\|$;
\item [(ii)]$F_{2}$ is compact.
\end{itemize}
Then $F$ is condensing.
\end{lemma}

\begin{lemma}\label{lemma2.7}(Sadovskii fixed point theorem)\cite{sad}
Let $Q$ be a convex, bounded, and closed
subset of a Banach space $X$ and let $ N: Q \rightarrow Q$ be a condensing
map. Then, $N$ has a fixed point in $Q$.
\end{lemma}

\section{The existence and boundedness of solution  }
\noindent
\setcounter{equation}{0}
In this section, we shall deal with the existence of solution of (\ref{eq2.4}) by using the Sadvoskii fixed point theorem and the Banach fixed point theorem. We will show  boundedness of solution of (\ref{eq2.4}) by employing the inequality techniques.

Note that
$$u(t)=u_0-\frac{1}{\Gamma(\alpha)}\int_{0}^a (a-s)^{\alpha-1}h(s)ds+\frac{1}{\Gamma(\alpha)}\int_{0}^t (t-s)^{\alpha-1}h(s)ds,$$
solves the Cauchy problems
\begin{equation*}
\left\{
\begin{array}{l}
\leftidx{_0^C}D{_t^\alpha}u(t)=h(t),\ t\in J, \\
u(0)=u_0-\frac{1}{\Gamma(\alpha)}\int_{0}^a (a-s)^{\alpha-1}h(s)ds.
\end{array}
\right.
\end{equation*}
One can obtain the following result immediately.

\begin{lemma}\label{lemma3.1}\cite{JRWang}
Let $0<\alpha<1$ and $h:J\rightarrow R$ be continuous. A function $u\in C(J,R)$ is a solution of the fractional integral equation
$$u(t)=u_0-\frac{1}{\Gamma(\alpha)}\int_{0}^a (a-s)^{\alpha-1}h(s)ds+\frac{1}{\Gamma(\alpha)}\int_{0}^t (t-s)^{\alpha-1}h(s)ds,$$
if and only if $u$ is a solution of the following fractional Cauchy problems
\begin{equation*}
\left\{
\begin{array}{l}
\leftidx{_0^C}D{_t^\alpha}u(t)=h(t),\ t\in J, \\
u(a)=u_0,\ a>0.
\end{array}
\right.
\end{equation*}
\end{lemma}

By using  Lemma \ref{lemma3.1}, we can obtain the following result.
\begin{theorem}\label{th3.1}
$x(t)$ is a solution of the neural networks (\ref{eq2.4}) with initial value $x(0)=x_0$, if and only if $x(t)$ is a solution of the following fractional integral equation:
\begin{equation}\label{eq3.1}
x(t)=
  \left\{
\begin{array}{l}
x_0-\frac{1}{\Gamma(\alpha)}\int_{0}^t (t-s)^{\alpha-1}x(s)ds+\frac{1}{\Gamma(\alpha)}\int_{0}^t (t-s)^{\alpha-1}P_{K}(x(s) -\rho Ax(s)-\rho b)ds,\ t\in [0,t_1], \\
x_0+U_1(x(t_1^{-}))-\frac{1}{\Gamma(\alpha)}\int_{0}^t (t-s)^{\alpha-1}x(s)ds\\
     ~~+\frac{1}{\Gamma(\alpha)}\int_{0}^t (t-s)^{\alpha-1}P_{K}(x(s) -\rho Ax(s)-\rho b)ds, t\in (t_1,t_2], \\
x_0+U_1(x(t_1^{-}))+U_2(x(t_2^{-}))-\frac{1}{\Gamma(\alpha)}\int_{0}^t (t-s)^{\alpha-1}x(s)ds \\
    ~~ +\frac{1}{\Gamma(\alpha)}\int_{0}^t (t-s)^{\alpha-1}P_{K}(x(s) -\rho Ax(s)-\rho b)ds,\ t\in (t_2,t_3], \\
\vdots \\
x_0+\sum \limits_{j=1}^{k}U_{j}(x(t_j^{-}))-\frac{1}{\Gamma(\alpha)}\int_{0}^t (t-s)^{\alpha-1}x(s)ds\\
    ~~ +\frac{1}{\Gamma(\alpha)}\int_{0}^t (t-s)^{\alpha-1}P_{K}(x(s) -\rho Ax(s)-\rho b)ds,\ t\in (t_k,t_{k+1}],\\
\vdots \\
x_0+\sum \limits_{j=1}^{m}U_{j}(x(t_j^{-}))-\frac{1}{\Gamma(\alpha)}\int_{0}^t (t-s)^{\alpha-1}x(s)ds\\
   ~~+\frac{1}{\Gamma(\alpha)}\int_{0}^t (t-s)^{\alpha-1}P_{K}(x(s) -\rho Ax(s)-\rho b)ds,\ t\in (t_m,T].
\end{array}
\right.
   \end{equation}
\end{theorem}
\begin{proof}
Necessity: Let $x(t)$ is a solution of the neural networks (\ref{eq2.4}) with initial value $x(0)=x_0$. For $t\in[0,t_1]$, it follows that
\begin{equation}\label{eq3.2}
\leftidx{_{0}^C}D{_t^\alpha}x(t)=-x(t)+P_{K}(x(t) -\rho Ax(t)-\rho b),\ t\in[0,t_1].
\end{equation}
Integrating  (\ref{eq3.2}) from $0$ to $t$ by virtue of Definition \ref{def2.1}, one can have
$$x(t)=x_0-\frac{1}{\Gamma(\alpha)}\int_{0}^t (t-s)^{\alpha-1}x(s)ds+\frac{1}{\Gamma(\alpha)}\int_{0}^t (t-s)^{\alpha-1}P_{K}(x(s) -\rho Ax(s)-\rho b)ds.$$
If $t\in(t_1,t_2]$, then
 $$\leftidx{_{0}^C}D{_t^\alpha}x(t)=-x(t)+P_{K}(x(t) -\rho Ax(t)-\rho b),\ t\in(t_1,t_2]  \quad with  \quad  x(t_1^{+})= x(t_1^{-})+U_1(x(t_1^{-})).$$
By the means of Lemma \ref{lemma3.1}, one obtain
\begin{align*}
x(t)&=x(t_1^{+})-\frac{1}{\Gamma(\alpha)}\int_{0}^{t_1} (t_1-s)^{\alpha-1}[-x(s)+P_{K}\left(x(s) -\rho Ax(s)-\rho b)\right]ds\notag\\
   & \quad +\frac{1}{\Gamma(\alpha)}\int_{0}^t (t-s)^{\alpha-1}[-x(s)+P_{K}(x(s) -\rho Ax(s)-\rho b)]ds\notag\\
  &=x(t_1^{-})+U_{1}(x(t_1^{-})) -\frac{1}{\Gamma(\alpha)}\int_{0}^{t_1} (t_1-s)^{\alpha-1}[-x(s)+P_{K}(x(s) -\rho Ax(s)-\rho b)]ds\notag\\
   & \quad +\frac{1}{\Gamma(\alpha)}\int_{0}^t (t-s)^{\alpha-1}[-x(s)+P_{K}(x(s) -\rho Ax(s)-\rho b)]ds\notag\\
  &=x_0+U_1(x(t_1^{-}))+\frac{1}{\Gamma(\alpha)}\int_{0}^t (t-s)^{\alpha-1}[-x(s)+P_{K}(x(s) -\rho Ax(s)-\rho b)]ds \notag\\
  &=x_0+U_1(x(t_1^{-}))-\frac{1}{\Gamma(\alpha)}\int_{0}^t (t-s)^{\alpha-1}x(s)ds+\frac{1}{\Gamma(\alpha)}\int_{0}^t (t-s)^{\alpha-1}P_{K}(x(s) -\rho Ax(s)-\rho b)ds.
\end{align*}
If $t\in(t_2,t_3]$, then using again Lemma \ref{lemma3.1}, we get
\begin{align*}
x(t)&=x(t_2^{+})-\frac{1}{\Gamma(\alpha)}\int_{0}^{t_2} (t_2-s)^{\alpha-1}[-x(s)+P_{K}(x(s) -\rho Ax(s)-\rho b)]ds\notag\\
   & \quad +\frac{1}{\Gamma(\alpha)}\int_{0}^t (t-s)^{\alpha-1}[-x(s)+P_{K}(x(s) -\rho Ax(s)-\rho b)]ds\notag\\
  &=x(t_2^{-})+U_{2}(x(t_2^{-})) -\frac{1}{\Gamma(\alpha)}\int_{0}^{t_2} (t_2-s)^{\alpha-1}(s)[-x(s)+P_{K}(x(s) -\rho Ax(s)-\rho b)]ds\notag\\
   & \quad +\frac{1}{\Gamma(\alpha)}\int_{0}^t (t-s)^{\alpha-1}[-x(s)+P_{K}(x(s) -\rho Ax(s)-\rho b)]ds\notag\\
  &=x_0+U_1(x(t_1^{-}))+U_2(x(t_2^{-}))+\frac{1}{\Gamma(\alpha)}\int_{0}^t (t-s)^{\alpha-1}[-x(s)+P_{K}(x(s) -\rho Ax(s)-\rho b)]ds \notag\\
  &=x_0+U_1(x(t_1^{-}))+U_2(x(t_2^{-}))-\frac{1}{\Gamma(\alpha)}\int_{0}^t (t-s)^{\alpha-1}x(s)ds\notag\\
  & \quad +\frac{1}{\Gamma(\alpha)}\int_{0}^t (t-s)^{\alpha-1}P_{K}(x(s) -\rho Ax(s)-\rho b)ds.
\end{align*}
If $t\in(t_k,t_{k+1}]$, $k=3,4,\cdots,m$, then again from Lemma \ref{lemma3.1}, we get
$$x(t)=x_0+\sum \limits_{j=1}^{k}U_{j}(x(t_j^{-}))-\frac{1}{\Gamma(\alpha)}\int_{0}^t (t-s)^{\alpha-1}x(s)ds+\frac{1}{\Gamma(\alpha)}\int_{0}^t (t-s)^{\alpha-1}P_{K}(x(s) -\rho Ax(s)-\rho b)ds.$$
Sufficiency: Assume that $x(t)$ satisfies (\ref{eq3.1}). If $t\in[0,t_1]$ then $x(0)=x_0$ and using that $\leftidx{_{0}^C}D{_t^\alpha}$ is the left inverse of
$\leftidx{_{0}}D{_t^{-\alpha}}$, we get (\ref{eq3.2}). If $t\in(t_k,t_{k+1}]$, $k=1,2,\cdots,m$ and using the fact of the Caputo fractional derivative of a constant is equal to zero, we obtain
$$\leftidx{_{0}^C}D{_t^\alpha}x(t)=-x(t)+P_{K}[x(t) -\rho Ax(t)-\rho b],\quad t\in(t_k,t_{k+1}]$$
and $x(t_k^{+})-x(t_k^{-})=U_{k}(x(t_k)).$ This completes the proof.
\end{proof}

Next, we will show the existence result concerned with the solution of (\ref{eq2.4})
by employing the Sadovskii fixed point theorem. To this end, we need the following
hypotheses:

$[H]: U_{k}:R^{n} \longrightarrow R^{n}$  and there exist $l_1,l_2>0$ such that
$$\|U_{k}(x)\|\leq l_1 ~ \text{and} ~ \|U_{k}(x)-U_{k}(y)\|\leq l_2\|x-y\|, ~ \text{for all} ~x,y \in PC(J,R^n).$$

\begin{theorem}\label{th3.2}
Under the assumption $[H]$, the neural networks (\ref{eq2.4}) has at least one  solution $x\in PC(J,R^n)$, provided that $\frac{1+\|I-\rho A\|}{\Gamma(\alpha+1)}T^{\alpha}<1$ and $ml_2+\frac{T^{\alpha}}{\Gamma(\alpha+1)}<1$.
\end{theorem}
\begin{proof}
Consider an operator $\Phi:PC(J,R^n)\rightarrow PC(J,R^n)$ defined by
\begin{align}\label{eq3.3}
(\Phi x)(t)&=x_0+\sum \limits_{j=1}^{k}U_{j}(x(t_j^{-}))-\frac{1}{\Gamma(\alpha)}\int_{0}^t (t-s)^{\alpha-1}x(s)ds\notag\\
           & \quad+\frac{1}{\Gamma(\alpha)}\int_{0}^t (t-s)^{\alpha-1}P_{K}(x(s) -\rho Ax(s)-\rho b)ds, t\in (t_{k},t_{k+1}],k=1,2,\cdots,m.
\end{align}
We define two operators as follows:

$F:PC(J,R^n)\rightarrow PC(J,R^n)$ given by
$$(Fx)(t)=x_0+\sum \limits_{j=1}^{k}U_{j}(x(t_{j}))-\frac{1}{\Gamma(\alpha)}\int_{0}^t (t-s)^{\alpha-1}x(s)ds, t\in (t_{k},t_{k+1}],k=1,2,\cdots,m.$$

$G:PC(J,R^n)\rightarrow PC(J,R^n)$ given by
$$(Gx)(t)=\frac{1}{\Gamma(\alpha)}\int_{0}^t (t-s)^{\alpha-1}P_{K}(x(s) -\rho Ax(s)-\rho b)ds.$$

Clearly, $\Phi=F+G.$

Let $B_{\delta}$ be a bounded closed and convex subset of $PC(J,R^n)$ and is defined as $B_{\delta}= \{x: \|x\|\leq \delta\}$, where
\begin{equation}\label{eq0}
\delta \geq\frac{\|x_0\|+ml_1+\frac{\|\rho b\|+\|\psi\|}{\Gamma(\alpha+1)}T^{\alpha}}{1-\frac{1+\|I-\rho A\|}{\Gamma(\alpha+1)}T^{\alpha}}, \psi\in K.
\end{equation}
Next, we prove that $\Phi$ has at least one fixed point. The proof is divided into
three steps.

Step 1. $\Phi$ is self mapping. Now we show that $\Phi(B_{\delta})\subset B_{\delta}$.

For any $x\in B_{\delta}$. If $t\in[0,t_1]$,  we get
\begin{equation}\label{eq3.4}
\big\|(\Phi x)(t) \big\| \leq \|x_0\|+\frac{1}{\Gamma(\alpha)} \int_{0}^{t} (t-s)^{\alpha-1} \|x(s)\|ds+\frac{1}{\Gamma(\alpha)}\int_{0}^{t} (t-s)^{\alpha-1}\big\|P_{K}(x(s) -\rho Ax(s)-\rho b)\big\|ds.
\end{equation}

According to the definition of projection operator, since $K$ is not empty, we can choose a point $\psi\in K$, which makes the following inequality hold:
$$\|P_{K}(0)\|\leq \|\psi\|.$$
By Lemma \ref{lemma2.3}, we have
\begin{align}\label{eq3.5}
\big\|P_{K}(x(s)-\rho Ax(s)-\rho b)\big\|&=\big\|P_{K}(x(s) -\rho Ax(s)-\rho b)-P_{K}(0)+P_{K}(0)\big\| \notag\\
&\leq \big\|P_{K}(x(s) -\rho Ax(s)-\rho b)-P_{K}(0)\big\|+\big\|P_{K}(0)\big\| \notag\\
&\leq \big\| x(s)-\rho Ax(s)-\rho b  \big\|+\|\psi\| \notag\\
&=\big\|(I-\rho A)x(s)-\rho b \big\|+\|\psi\|  \notag\\
&\leq\|I-\rho A\|\|x(s)\|+\|\rho b\|+\|\psi\|,
\end{align}
where $I$ stands for the identity matrix with appropriate dimension.

It follows from (\ref{eq0}), (\ref{eq3.4}) and (\ref{eq3.5}) that
\begin{align}\label{eq3.6}
\big\|(\Phi x)(t) \big\| &\leq  \|x_0\|+\frac{1}{\Gamma(\alpha)} \int_{0}^{t} (t-s)^{\alpha-1} \|x(s)\|ds\notag\\
        & \quad+\frac{1}{\Gamma(\alpha)}\int_{0}^{t} (t-s)^{\alpha-1}(\|I-\rho A\|\|x(s)\|+\|\rho b\|+\|\psi\|)ds\notag\\
& =\|x_0\|+\frac{1+\|I-\rho A\|}{\Gamma(\alpha)} \int_{0}^{t} (t-s)^{\alpha-1} \|x(s)\|ds+\frac{\|\rho b\|+\|\psi\|}{\Gamma(\alpha)}\int_{0}^{t} (t-s)^{\alpha-1}ds\notag\\
     &\leq \|x_0\|+\frac{(1+\|I-\rho A\|)\delta+\|\rho b\|+\|\psi\|}{\Gamma(\alpha+1)}T^{\alpha}\leq\delta.
\end{align}
If $t\in(t_1,t_2]$, by (\ref{eq0}), (\ref{eq3.4}) and (\ref{eq3.5}), we have
\begin{align}\label{eq3.7}
\big\|(\Phi x)(t) \big\| &\leq \|x_0\|+\|U_1(x(t_1^{-}))\|+\frac{1}{\Gamma(\alpha)} \int_{0}^{t} (t-s)^{\alpha-1} \|x(s)\|ds\notag\\
    & \quad+\frac{1}{\Gamma(\alpha)}\int_{0}^{t} (t-s)^{\alpha-1}\big\|P_{K}(x(s) -\rho Ax(s)-\rho b)\big\|ds\notag\\
&\leq  \|x_0\|+\|U_1(x(t_1^{-}))\|+\frac{1}{\Gamma(\alpha)} \int_{0}^{t} (t-s)^{\alpha-1} \|x(s)\|ds\notag\\
       & \quad+\frac{1}{\Gamma(\alpha)}\int_{0}^{t} (t-s)^{\alpha-1}(\|(I-\rho A)\|\|x(s)\|+\|\rho b\|+\|\psi\|)ds\notag\\
& =\|x_0\|+l_1+\frac{1+\|I-\rho A\|}{\Gamma(\alpha)} \int_{0}^{t} (t-s)^{\alpha-1} \|x(s)\|ds+\frac{\|\rho b\|+\|\psi\|}{\Gamma(\alpha)}\int_{0}^{t} (t-s)^{\alpha-1}ds\notag\\
     &\leq \|x_0\|+l_1+\frac{(1+\|I-\rho A\|)\delta+\|\rho b\|+\|\psi\|}{\Gamma(\alpha+1)}T^{\alpha} \leq\delta.
\end{align}
If $t\in(t_k,t_{k+1}],k=3,4,\cdots,m$, using the same argument, we get
\begin{align}\label{eq3.7}
\big\|(\Phi x)(t) \big\| &\leq \|x_0\|+\sum \limits_{j=1}^{m}\|U_{j}(x(t_j^{-}))\|+\frac{1}{\Gamma(\alpha)} \int_{0}^{t} (t-s)^{\alpha-1} \|x(s)\|ds\notag\\
      & \quad+\frac{1}{\Gamma(\alpha)}\int_{0}^{t} (t-s)^{\alpha-1}\big\|P_{K}(x(s) -\rho Ax(s)-\rho b)\big\|ds\notag\\
&\leq  \|x_0\|+ml_1+\frac{1}{\Gamma(\alpha)} \int_{0}^{t} (t-s)^{\alpha-1} \|x(s)\|ds\notag\\
       & \quad+\frac{1}{\Gamma(\alpha)}\int_{0}^{t} (t-s)^{\alpha-1}(\|I-\rho A\|\|x(s)\|+\|\rho b\|+\|\psi\|)ds\notag\\
& =\|x_0\|+ml_1+\frac{1+\|I-\rho A\|}{\Gamma(\alpha)} \int_{0}^{t} (t-s)^{\alpha-1} \|x(s)\|ds+\frac{\|\rho b\|+\|\psi\|}{\Gamma(\alpha)}\int_{0}^{t} (t-s)^{\alpha-1}ds\notag\\
     &\leq \|x_0\|+ml_1+\frac{(1+\|I-\rho A\|)\delta+\|\rho b\|+\|\psi\|}{\Gamma(\alpha+1)}T^{\alpha}\leq\delta,
\end{align}
which implies that $\Phi(B_{\delta})\subset B_{\delta}.$

Step 2 $F$ is $k_1$-contraction.
If $t\in[0,t_1]$, for $x,y\in PC(J,R^n)$, we have
\begin{align*}
\|(Fx)(t)-(Fy)(t)\|&\leq \frac{1}{\Gamma(\alpha)} \int_{0}^{t} (t-s)^{\alpha-1} \|x(s)-y(s)\|ds\\
                   &\leq \frac{1}{\Gamma(\alpha)} \int_{0}^{t} (t-s)^{\alpha-1} ds\|x-y\|\\
                     &\leq \frac{T^{\alpha}}{\Gamma(\alpha+1)}\|x-y\|.
\end{align*}
If $t\in(t_1,t_2]$,  for $x,y\in PC(J,R^n)$, we have
\begin{align*}
\|(Fx)(t)-(Fy)(t)\|&\leq \|U_{1}(x(t_1^-))-U_{1}(y(t_1^-))\|+\frac{1}{\Gamma(\alpha)} \int_{0}^{t} (t-s)^{\alpha-1} \|x(s)-y(s)\|ds\\
                     &\leq \left(l_2+\frac{T^{\alpha}}{\Gamma(\alpha+1)}\right)\|x-y\|.
\end{align*}
If $t\in(t_k,t_{k+1}],k=3,4,\cdots,m$, for $x,y\in PC(J,R^n)$, using the same argument, we get,
\begin{align*}
\|(Fx)(t)-(Fy)(t)\|&\leq \sum \limits_{j=1}^{m}\|U_{j}(x(t_j^-))-U_{j}(y(t_j^-))\|+\frac{1}{\Gamma(\alpha)} \int_{0}^{t} (t-s)^{\alpha-1} \|x(s)-y(s)\|ds\\
                     &\leq \left(ml_2+\frac{T^{\alpha}}{\Gamma(\alpha+1)}\right)\|x-y\|\\
                     &= k_1\|x-y\|,
\end{align*}
where
$$
k_1=ml_2+\frac{T^{\alpha}}{\Gamma(\alpha+1)}.
$$

Thus, for $t\in[0,T]$, $x,y\in PC(J,R^n)$, one has
$$\|Fx-Fy\|_{PC}=\sup_{t\in J}\big\|(Fx)(t)-(Fy)(t)\big\|\leq k_1\|x-y\|_{PC},$$
and so $F$ is a contraction operator on $PC(J,R^n)$ due to $0<k_1<1$.

Step 3 $G$ is compact.

Claim 1. $G(B_{\delta})$ is bounded.

 For any $t\in [0,T]$, for $x\in B_{\delta}$, by (\ref{eq3.5}),we get
\begin{align*}
\|(Gx)(t)\|&\leq \frac{1}{\Gamma(\alpha)} \int_{0}^{t} (t-s)^{\alpha-1} \|P_{K}(x_n(s) -\rho Ax_n(s)-\rho b)\|ds\\
                   &\leq \frac{1}{\Gamma(\alpha)} \int_{0}^{t} (t-s)^{\alpha-1}(\|I-\rho A\|\|x_n(s)\|+\|\rho b\|+\|\psi\|)ds.
\end{align*}
Thus, one has
\begin{align*}
\|Gx\|&\leq\frac{T^{\alpha}}{\Gamma(\alpha+1)}(\|I-\rho A\|\|x\|+\|\rho b\|+\|\psi\|)\\
                   &\leq \frac{T^{\alpha}}{\Gamma(\alpha+1)}(\|I-\rho A\|\delta+\|\rho b\|+\|\psi\|)
\end{align*}
and so $G(B_{\delta})$ is bounded.

Claim 2. $Gx$ is equicontinuous.

For $0\leq\tau_1<\tau_2<T, x\in B_{\delta}$, based on (\ref{eq3.5}), we get
\begin{align*}
\|(Gx)(\tau_1)-(Gx)(\tau_2)\|&=\big\|\frac{1}{\Gamma(\alpha)} \int_{0}^{\tau_1} (\tau_1-s)^{\alpha-1} P_{K}(x(s) -\rho Ax(s)-\rho b)ds\\
                      & \quad -\frac{1}{\Gamma(\alpha)} \int_{0}^{\tau_2} (\tau_2-s)^{\alpha-1} P_{K}(x(s) -\rho Ax(s)-\rho b)ds\big\|\\
                   &\leq \frac{1}{\Gamma(\alpha)} \int_{0}^{\tau_1} [(\tau_1-s)^{\alpha-1}-(\tau_2-s)^{\alpha-1}]\|P_{K}(x(s) -\rho Ax(s)-\rho b)\| ds\\
                    & \quad +\frac{1}{\Gamma(\alpha)} \int_{\tau_1}^{\tau_2} (\tau_2-s)^{\alpha-1}\|P_{K}(x(s) -\rho Ax(s)-\rho b)\| ds\\
                    &\leq \frac{1}{\Gamma(\alpha)} \int_{0}^{\tau_1} [(\tau_1-s)^{\alpha-1}-(\tau_2-s)^{\alpha-1}]\left(\|I-\rho A\|\|x_n(s)\|+\|\rho b\|+\|\psi\|\right)ds\\
                    & \quad +\frac{1}{\Gamma(\alpha)} \int_{\tau_1}^{\tau_2} (\tau_2-s)^{\alpha-1}(\|I-\rho A\|\|x_n(s)\|+\|\rho b\|+\|\psi\|) ds\\
                    &\leq \frac{\|I-\rho A\|\eta+\|\rho b\|+\|\psi\|}{\Gamma(\alpha)}\left(\frac{\tau_1^{\alpha}}{\alpha}-\frac{\tau_2^{\alpha}}{\alpha}+\frac{(\tau_2-\tau_1)^{\alpha}}{\alpha}+\frac{(\tau_2-\tau_1)^{\alpha}}{\alpha}\right)\\
                    &\leq \frac{2(\|I-\rho A\|\eta+\|\rho b\|+\|\psi\|)(\tau_2-\tau_1)^{\alpha}}{\alpha\Gamma(\alpha)}.
\end{align*}
As $\tau_2\rightarrow\tau_1$, the right-hand side of the above inequality tends to zero. Therefore, $Gx$ is equicontinuous. Thus using Arzela-Ascoli theorem for equicontinuous functions, we conclude that
$G(B_{\delta})$ is relatively compact.

As $\Phi=F+G$, $F$ is continuous and $\kappa_1$-contraction, $G$ is compact, so by using Lemma \ref{lemma2.6}, $\Phi$ is condensing map. Therefore, we show that mapping $\Phi$ satisfy all the conditions of sadvoskii fixed point theorem. Consequently, by Lemma \ref{lemma2.7} we deduce that $\Phi$ has at least one
fixed point. This completes the proof.
\end{proof}

\begin{theorem}\label{th3.2}
Under the assumption $[H]$, the neural networks (\ref{eq2.4}) has a unique solution $x\in PC(J,R^n)$, provided that $m l_2+\frac{1+\|I-\rho A\|}{\Gamma(\alpha+1)}T^{\alpha}<1$.
\end{theorem}
\begin{proof}
Consider an operator $\Phi:PC(J,R^n)\rightarrow PC(J,R^n)$ defined by (\ref{eq3.3}). Next, we will show the uniqueness result concerned with the solution of (\ref{eq2.4})
by employing the Banach fixed point theorem. The proof is divided into
two steps.

Step 1. $\Phi x\in PC(J,R^n)$ for every $x\in PC(J,R^n)$.

Let $\varpi_k=\max\|x(t)\|,t\in (t_k,t_{k+1}], k=0, 1, 2, \cdots,m.$
If $t\in[0,t_1]$, for every $x\in C([0,t_1],R^n)$ and any $\epsilon>0$ satisfy $t<t+\epsilon<t_1$, we get
\begin{align}\label{eq3.9}
&\big\|(\Phi x)(t+\epsilon)-(\Phi x)(t) \big\| \notag\\
=&\bigg\| -\frac{1}{\Gamma(\alpha)}\int_{0}^{t+\epsilon} (t+\epsilon-s)^{\alpha-1}x(s)ds+\frac{1}{\Gamma(\alpha)}\int_{0}^{t+\epsilon} (t+\epsilon-s)^{\alpha-1}P_{K}(x(s) -\rho Ax(s)-\rho b)ds\notag\\
& \quad +\frac{1}{\Gamma(\alpha)}\int_{0}^{t} (t-s)^{\alpha-1}x(s)ds-\frac{1}{\Gamma(\alpha)}\int_{0}^{t} (t-s)^{\alpha-1}P_{K}(x(s) -\rho Ax(s)-\rho b)ds\bigg\| \notag\\
\leq& \frac{1}{\Gamma(\alpha)} \int_{0}^{t} -[(t+\epsilon-s)^{\alpha-1}-(t-s)^{\alpha-1}] \|x(s)\|ds+\frac{1}{\Gamma(\alpha)} \int_{t}^{t+\epsilon} (t+\epsilon-s)^{\alpha-1}\|x(s)\|ds\notag\\
    & \quad +\frac{1}{\Gamma(\alpha)}\int_{0}^{t} -[(t+\epsilon-s)^{\alpha-1}-(t-s)^{\alpha-1}]\big\|P_{K}(x(s) -\rho Ax(s)-\rho b)\big\|ds \notag\\
     & \quad +\frac{1}{\Gamma(\alpha)}\int_{t}^{t+\epsilon} (t+\epsilon-s)^{\alpha-1}\big\|P_{K}(x(s) -\rho Ax(s)-\rho b)\big\|ds.
\end{align}

It follows from (\ref{eq3.5}) and (\ref{eq3.9}) that
\begin{align}\label{eq3.6}
\big\|(\Phi x)(t+\epsilon)-(\Phi x)(t) \big\| &\leq \frac{1}{\Gamma(\alpha)} \int_{0}^{t} -[(t+\epsilon-s)^{\alpha-1}-(t-s)^{\alpha-1}] (1+\|I-\rho A\|)\|x(s)\|ds\notag\\
 & \quad +\frac{1}{\Gamma(\alpha)} \int_{t}^{t+\epsilon} (t+\epsilon-s)^{\alpha-1}(1+\|I-\rho A\|)\|x(s)\|ds\notag\\
    & \quad +\frac{\rho\|b\|+\|\psi\|}{\Gamma(\alpha)}\int_{0}^{t} -[(t+\epsilon-s)^{\alpha-1}-(t-s)^{\alpha-1}]ds \notag\\
     & \quad +\frac{\rho\|b\|+\|\psi\|}{\Gamma(\alpha)}\int_{t}^{t+\epsilon} (t+\epsilon-s)^{\alpha-1}ds\notag\\
     &\leq \frac{(1+\|I-\rho A\|)\varpi_0}{\Gamma(\alpha)} \int_{0}^{t} -[(t+\epsilon-s)^{\alpha-1}-(t-s)^{\alpha-1}] ds\notag\\
    & \quad +\frac{(1+\|I-\rho A\|)\varpi_0}{\Gamma(\alpha)} \int_{t}^{t+\epsilon} (t+\epsilon-s)^{\alpha-1}ds\notag\\
    & \quad +\frac{\rho\|b\|+\|\psi\|}{\Gamma(\alpha)}\int_{0}^{t} -[(t+\epsilon-s)^{\alpha-1}-(t-s)^{\alpha-1}]ds \notag\\
     & \quad +\frac{\rho\|b\|+\|\psi\|}{\Gamma(\alpha)}\int_{t}^{t+\epsilon} (t+\epsilon-s)^{\alpha-1}ds\notag\\
     &=\frac{(1+\|I-\rho A\|)\varpi_0+\|\rho b\|+\|\psi\|}{\Gamma(\alpha)} \left(\frac{\varepsilon^{\alpha}}{\alpha}-\frac{(t+\varepsilon)^{\alpha}}{\alpha}+\frac{t^{\alpha}}{\alpha}\right)\notag\\
     & \quad +\frac{(1+\|I-\rho A\|)\varpi_0+\|\rho b\|+\|\psi\|}{\Gamma(\alpha+1)}\varepsilon^{\alpha}\notag\\
     &\leq 2\frac{(1+\|I-\rho A\|)\varpi_0+\|\rho b\|+\|\psi\|}{\Gamma(\alpha+1)}\varepsilon^{\alpha}.
\end{align}
It is easy to see that the right-hand of the above inequality tends to zero as $\epsilon\rightarrow 0$. Thus, $\Phi x\in C([0,t_1],R^n).$

If $t\in(t_1,t_2]$, for every $x\in C((t_1,t_2],R^n)$ and any $\epsilon>0$ satisfy $t_1<t+\epsilon<t_2$, repeating the same procedures, we get
\begin{align*}
\big\|(\Phi x)(t+\epsilon)-(\Phi x)(t) \big\|&\leq\frac{(1+\|I-\rho A\|)\varpi_1+\rho\|b\|+\|\psi\|}{\Gamma(\alpha)} \left(\frac{\varepsilon^{\alpha}}{\alpha}-\frac{(t+\varepsilon)^{\alpha}}{\alpha}+\frac{t^{\alpha}}{\alpha}\right)\notag\\
     & \quad     +l_2\|x((t+\epsilon)_1^{-})-x(t_1^{-})\|+\frac{(1+\|I-\rho A\|)\varpi_1+\rho\|b\|+\|\psi\|}{\Gamma(\alpha+1)}\varepsilon^{\alpha}.
\end{align*}
It is easy to see that the right-hand of the above inequality tends to zero as $\epsilon\rightarrow 0$. Thus, $\Phi x\in C((t_1,t_2],R^n).$

If $t\in(t_k,t_{k+1}], k=1,2,\cdots, m, $ for every $x\in C((t_k,t_{k+1}], R^n)$ and any $\epsilon>0$ satisfy $t_k<t+\epsilon<t_{k+1}$, using the same process again, we get $\Phi x\in C((t_k,t_{k+1}], R^n).$

Step 2. $\Phi$ is a contraction operator on $PC(J, R^n)$.

If $t\in[0,t_1]$, for $x,y\in PC(J,R^n)$, by Lemma \ref{lemma2.3},we have
\begin{align*}
\big\|(\Phi x)(t)-(\Phi y)(t)\big\|&\leq\frac{1}{\Gamma(\alpha)} \int_{0}^{t} (t-s)^{\alpha-1}\big\|P_{K}(x(s) -\rho Ax(s)-\rho b)-P_{K}(y(s) -\rho Ay(s)-\rho b)  \big\|ds\notag\\
     & \quad +\frac{1}{\Gamma(\alpha)} \int_{0}^{t} (t-s)^{\alpha-1}\|x(s)-y(s)\|ds\notag\\
     &\leq\frac{1}{\Gamma(\alpha)} \int_{0}^{t} (t-s)^{\alpha-1}\big\|\left(x(s) -\rho Ax(s)-\rho b\right)-\left(y(s) -\rho Ay(s)-\rho b\right) \big \|ds\notag\\
     & \quad+\frac{1}{\Gamma(\alpha)} \int_{0}^{t} (t-s)^{\alpha-1}\|x(s)-y(s)\|ds\notag\\
     &\leq \frac{1+\|I-\rho A\|}{\Gamma(\alpha)} \int_{0}^{t} (t-s)^{\alpha-1}\|x(s)-y(s)\|ds\notag\\
     &\leq \frac{1+\|I-\rho A\|}{\Gamma(\alpha)}\|x-y)\|_{PC}\int_{0}^{t} (t-s)^{\alpha-1}ds\notag\\
     &\leq\left(\frac{1+\|I-\rho A\|}{\Gamma(\alpha+1)}T^{\alpha}\right)\|x-y\|_{PC}.
\end{align*}

If $t\in(t_1,t_2]$,  for $x,y\in PC(J,R^n)$, we have
\begin{align*}
\big\|(\Phi x)(t)-(\Phi y)(t)\big\|&\leq\frac{1}{\Gamma(\alpha)} \int_{0}^{t} (t-s)^{\alpha-1}\big\|P_{K}(x(s) -\rho Ax(s)-\rho b)-P_{K}(y(s) -\rho Ay(s)-\rho b)  \big\|ds\notag\\
     & \quad +\|U_{1}(x(t_1^-))-U_{1}(y(t_1^-))\|+\frac{1}{\Gamma(\alpha)} \int_{0}^{t} (t-s)^{\alpha-1}\|x(s)-y(s)\|ds\notag\\
     &\leq\frac{1}{\Gamma(\alpha)} \int_{0}^{t} (t-s)^{\alpha-1}\big\|\left(x(s) -\rho Ax(s)-\rho b\right)-\left(y(s) -\rho Ay(s)-\rho b\right) \big \|ds\notag\\
     & \quad +\|U_{1}(x(t_1^-))-U_{2}(y(t_2^-))\|+\frac{1}{\Gamma(\alpha)} \int_{0}^{t} (t-s)^{\alpha-1}\|x(s)-y(s)\|ds\notag\\
     &\leq l_2\|x(t_1^-)-y(t_1^-)\|+\frac{1+\|I-\rho A\|}{\Gamma(\alpha)} \int_{0}^{t} (t-s)^{\alpha-1}\|x(s)-y(s)\|ds\notag\\
     &\leq  l_2\|x-y\|_{PC}+\frac{1+\|I-\rho A\|}{\Gamma(\alpha)}\|x-y)\|_{PC}\int_{0}^{t} (t-s)^{\alpha-1}ds\notag\\
     &\leq\left(l_2+\frac{1+\|I-\rho A\|}{\Gamma(\alpha+1)}T^{\alpha}\right)\|x-y\|_{PC}.
\end{align*}

If $t\in(t_k,t_{k+1}],k=3,4,\cdots,m$, using the same argument, we have
\begin{align}\label{eq3.7}
\big\|(\Phi x)(t)-(\Phi y)(t)\big\|&\leq\frac{1}{\Gamma(\alpha)} \int_{0}^{t} (t-s)^{\alpha-1}\big\|P_{K}(x(s) -\rho Ax(s)-\rho b)-P_{K}(y(s) -\rho Ay(s)-\rho b)  \big\|ds\notag\\
     & \quad +\sum \limits_{j=1}^{m}\|U_{j}(x(t_j^-))-U_{j}(y(t_j^-))\|+\frac{1}{\Gamma(\alpha)} \int_{0}^{t} (t-s)^{\alpha-1}\|x(s)-y(s)\|ds\notag\\
     &\leq\frac{1}{\Gamma(\alpha)} \int_{0}^{t} (t-s)^{\alpha-1}\big\|\left(x(s) -\rho Ax(s)-\rho b\right)-\left(y(s) -\rho Ay(s)-\rho b\right) \big \|ds\notag\\
     & \quad +\sum \limits_{j=1}^{m}\|U_{j}(x(t_j^-))-U_{j}(y(t_j^-))\|+\frac{1}{\Gamma(\alpha)} \int_{0}^{t} (t-s)^{\alpha-1}\|x(s)-y(s)\|ds\notag\\
     &\leq l_2\sum \limits_{j=1}^{m}\|x(t_j^-)-y(t_j^-)\|+\frac{1+\|I-\rho A\|}{\Gamma(\alpha)} \int_{0}^{t} (t-s)^{\alpha-1}\|x(s)-y(s)\|ds\notag\\
     &\leq m l_2\|x-y\|_{PC}+\frac{1+\|I-\rho A\|}{\Gamma(\alpha)}\|x-y\|_{PC}\int_{0}^{t} (t-s)^{\alpha-1}ds\notag\\
     &\leq\left(m l_2+\frac{1+\|I-\rho A\|}{\Gamma(\alpha+1)}T^{\alpha}\right)\|x-y\|_{PC}\notag\\
     &\leq \kappa_{2}\|x-y\|_{PC},
\end{align}
where
$$
\kappa_2=m l_2+\frac{1+\|I-\rho A\|}{\Gamma(\alpha+1)}T^{\alpha}.
$$
Thus, for $t\in[0,T]$, $x,y\in PC(J,R^n)$, one has
$$\|\Phi x-\Phi y\|_{PC}=\sup_{t\in J}\big\|(\Phi x)(t)-(\Phi y)(t)\big\|\leq\kappa_2 \|x-y\|_{PC}$$
and so $\Phi$ is a contraction operator on $PC(J,R^n)$ due to $0<\kappa_2<1$.  Now the Banach fixed point theorem shows that the operation $\Phi$ has a unique fixed point on $PC(J,R^n)$. This means that system \eqref{eq2.4} has a unique solution.
\end{proof}

\begin{theorem}\label{th3.3}
Under the assumption [H], the set of all solutions of fractional-order neural networks (\ref{eq2.4}) is bounded.
\end{theorem}
\begin{proof}
If $t\in[0,t_1]$, it follows from (\ref{eq3.1}), (\ref{eq3.5}) that
\begin{align*}
\|x(t)\|&\leq \|x_0\|+\frac{(1+\|I-\rho A\|)}{\Gamma(\alpha)} \int_{0}^{t} (t-s)^{\alpha-1}\|x(s)\|ds+\frac{\rho\|b\|+\|\psi\|}{\Gamma(\alpha)}\int_{0}^{t} (t-s)^{\alpha-1}ds\notag\\
     &\leq \|x_0\|+\frac{(\rho\|b\|+\|\psi\|)t^{\alpha}}{\Gamma(\alpha+1)}+\frac{(1+\|I_n-\rho A\|)}{\Gamma(\alpha)} \int_{0}^{t} (t-s)^{\alpha-1}\|x(s)\|ds.
\end{align*}

If $t\in(t_1,t_2]$, then we have
\begin{align*}
\|x(t)\|&\leq \|x_0\|+\|U_{1}(x(t_1^{-}))\|+\frac{(1+\|I-\rho A\|)}{\Gamma(\alpha)} \int_{0}^{t} (t-s)^{\alpha-1}\|x(s)\|ds+\frac{\rho\|b\|+\|\psi\|}{\Gamma(\alpha)}\int_{0}^{t} (t-s)^{\alpha-1}ds\notag\\
     &\leq \|x_0\|+l_1+\frac{(\rho\|b\|+\|\psi\|)t^{\alpha}}{\Gamma(\alpha+1)}+\frac{(1+\|I_n-\rho A\|)}{\Gamma(\alpha)} \int_{0}^{t} (t-s)^{\alpha-1}\|x(s)\|ds.
\end{align*}

If $t\in(t_k,t_{k+1}],k=3,4,\cdots,m$, arguing as the above, we have
\begin{align*}
\|x(t)\|&\leq \|x_0\|+\sum \limits_{j=1}^{m}\|U_{j}(x(t_j^{-}))\|+\frac{(1+\|I-\rho A\|)}{\Gamma(\alpha)} \int_{0}^{t} (t-s)^{\alpha-1}\|x(s)\|ds+\frac{\rho\|b\|+\|\psi\|}{\Gamma(\alpha)}\int_{0}^{t} (t-s)^{\alpha-1}ds\notag\\
     &\leq \|x_0\|+ml_1+\frac{(\rho\|b\|+\|\psi\|)t^{\alpha}}{\Gamma(\alpha+1)}+\frac{(1+\|I_n-\rho A\|)}{\Gamma(\alpha)} \int_{0}^{t} (t-s)^{\alpha-1}\|x(s)\|ds.
\end{align*}

Thus, for $t\in[0,T]$, one has
$$\|x(t)\|\leq\|x_0\|+ml_1+\frac{(\rho\|b\|+\|\psi\|)t^{\alpha}}{\Gamma(\alpha+1)}+\frac{(1+\|I_n-\rho A\|)}{\Gamma(\alpha)} \int_{0}^{t} (t-s)^{\alpha-1}\|x(s)\|ds$$
Let
$$a(t)=\|x_0\|+ml_1+\frac{(\rho\|b\|+\|\psi\|)t^{\alpha}}{\Gamma(\alpha+1)}, \quad b(t)=\frac{(1+\|I_n-\rho A\|)}{\Gamma(\alpha)}.$$
Obviously, $a(t)$ is a nonnegative and nondecreasing function. It follows from Lemma \ref{lemma2.5} that
\begin{align*}
\|x(t)\|&\leq a(t)E_\alpha(b(t)\Gamma(\alpha)t^\alpha)\notag\\
&\leq \left(\|x_0\|+ml_1+\frac{(\rho\|b\|+\|\psi\|)t^{\alpha}}{\Gamma(\alpha+1)}\right)E_\alpha\big(b(t)\Gamma(\alpha)t^\alpha\big).
\end{align*}
Thus, the set of all solutions of system (\ref{eq2.4}) is bounded. This completed the proof.
\end{proof}

\section{The Mittage-Leffler Stability of the equilibrium point }
\noindent
\setcounter{equation}{0}
In this section, we will show that the equilibrium point of the neural networks (\ref{eq2.4}) is globally Mittag-Leffler stable under some mild conditions.
\begin{theorem}\label{th4.1}
Assume that the impulse operator $U_k(x(t_k))$ satisfies
$$U_k(x(t_k))=-\sigma(x(t_k)-x^*),\quad k=1,2,\cdots,m.$$
Moreover, suppose that there exist a symmetric and positive definite matrix $Q > 0$ and positive scalars $ \rho_1>0$, $0<\eta_1\leq1$ such that
\begin{equation}\label{eq4.1}
  \lambda_{min}\left(-Q^{-\frac{1}{2}}\left(-2Q+\rho_1^{-1}Q^2+\rho_1(I-\rho A)^{T}(I-\rho A)\right) Q^{-\frac{1}{2}}\right)>0
   \end{equation}
and
\begin{equation}\label{eq4.2}
   Q^{-\frac{1}{2}}(I-\sigma)^{T}Q(I-\sigma) Q^{-\frac{1}{2}}-\eta_1 I\leq 0.
\end{equation}
Then the neural networks (\ref{eq2.4}) is globally Mittag-Leffler stable.
\end{theorem}
\begin{proof}
Assume that $x^*$ be the equilibrium point of the neural networks (\ref{eq2.4}) with initial value $x(0)=x_0$. Let $y(t)=x(t)-x^*$. Then the neural networks (\ref{eq2.4}) is transformed into
the follwing system:
\begin{equation} \label{eq4.3}
\left\{
\begin{array}{l}
\leftidx{_0^C}D{_t^\alpha}y(t)=-y(t)+\tilde{P}_{K}(y(t)),\ t\in J', \\
\triangle y(t_k) =y(t_k^{+})-y(t_k^{-})=-\sigma (y(t_k)),\\
y(0)=y_0,k=1,2,\cdots,m,
\end{array}
\right.
\end{equation}
where
$$\tilde{P}_{K}(y(t))=P_{K}[y(t)-\rho Ay(t)+x^* -\rho Ax^*-\rho b]-P_{K}[x^* -\rho Ax^*-\rho b], y(0)=x_0-x^*.$$

Clearly, the equilibrium point $x^{*}$  of (\ref{eq2.4}) is globally Mittag-Leffler stable if and only if the zero solution of system (\ref{eq4.3}) is
globally Mittag-Leffler stable.

Consider a Lyapunov function candidate:
$$V(t)=y^{T}(t)Qy(t).$$

When $t\neq t_k$, by calculating the $\alpha$-order Caputo derivatives of $V(t)$ along the trajectories of system (\ref{eq4.3}), we can obtain from Lemmas \ref{lemma2.1}, Lemmas \ref{lemma2.2} and Lemmas \ref{lemma2.3} that
\begin{align}\label{eq4.4}
\leftidx{_{ 0}^C}D{_t^{\alpha}}V(t)&=\leftidx{_{ 0}^C}D{_t^{\alpha}}\big(y^{T}(t)Qy(t)\big)\notag\\
                                   &\leq 2y^{T}(t)Q \left(\leftidx{_{ 0}^C}D{_t^{\alpha}}y(t)\right)\notag\\
                                   &=2y^{T}(t)Q\left(-y(t)+\tilde{P}_{K}(y(t))\right)\notag\\
                                   &=y^{T}(t)(-2Q)y(t)+2y^{T}(t)Q\tilde{P}_{K}(y(t))\notag\\
                                   &\leq y^{T}(t)(-2Q)y(t)+\rho_1^{-1}y^{T}(t)Q^2y(t)+\rho_1\tilde{P}_{K}^{T}(y(t))\tilde{P}_{K}(y(t))\notag\\
                                   &=y^{T}(t)(-2Q)y(t)+\rho_1^{-1}y^{T}(t)Q^2y(t)+\rho_1\|\tilde{P}_{K}(y(t))\|^2\notag\\
                                   &\leq y^{T}(t)(-2Q)y(t)+\rho_1^{-1}y^{T}(t)Q^2y(t)+\rho_1\|y(t)-\rho Ay(t)\|^2\notag\\
                                   &=y^{T}(t)(-2Q)y(t)+\rho_1^{-1}y^{T}(t)Q^2y(t)+\rho_1\left(y(t)-\rho Ay(t)\right)^{T}\left(y(t)-\rho Ay(t)\right)\notag\\
                                   &=y^{T}(t)(-2Q)y(t)+\rho_1^{-1}y^{T}(t)Q^2y(t)+\rho_1y^{T}(t)(I-\rho A)^{T}(I-\rho A)y(t)\notag\\
                                   &=y^{T}(t)\left(-2Q+\rho_1^{-1}Q^2+\rho_1(I-\rho A)^{T}(I-\rho A)\right)y(t)\notag\\
                                   &=y^{T}(t)Q^{\frac{1}{2}}\left(Q^{-\frac{1}{2}}\Pi Q^{-\frac{1}{2}}\right)Q^{\frac{1}{2}}y(t)\notag\\
                                   &=-y^{T}(t)Q^{\frac{1}{2}}\left(-Q^{-\frac{1}{2}}\Pi Q^{-\frac{1}{2}}\right)Q^{\frac{1}{2}}y(t),
\end{align}
where
$$\Pi=-2Q+\rho_1^{-1}Q^2+\rho_1(I-\rho A)^{T}(I-\rho A).$$

Let $\xi_1=\lambda_{min}\left(-Q^{-\frac{1}{2}}\Pi Q^{-\frac{1}{2}}\right).$ By using $\lambda_{min}(Q)\|y(t)\|^2\leq y^{T}(t)Qy(t)\leq \lambda_{max}(Q)\|y(t)\|^2$ and (\ref{eq4.1}), we can obtain
$$\leftidx{_{ 0}^C}D{_t^{\alpha}}V(t)\leq -y^{T}(t)Q^{\frac{1}{2}}\left(-Q^{-\frac{1}{2}}\Pi Q^{-\frac{1}{2}}\right)Q^{\frac{1}{2}}y(t)
                                     \leq -\lambda_{min}\left(-Q^{-\frac{1}{2}}\Pi Q^{-\frac{1}{2}}\right)y^{T}(t)Qy(t)=-\xi_1 V(t).$$
When $t=t_k$, it follows from ( \ref{eq4.2} ) that
\begin{align*}
V(t_k^{+})&=y^{T}(t_k^{+})Qy(t_k^{+})\notag\\
    &=\left(y(t_k^{-})-\sigma y(t_k^{-})\right)^{T}Q\left(y(t_k^{-})-\sigma y(t_k^{-})\right)\notag\\
    &=y^{T}(t_k)(I_n-\sigma)^{T}Q(I_n-\sigma)y(t_k)\notag\\
    &=y^{T}(t_k)Q^{\frac{1}{2}}\left(Q^{-\frac{1}{2}}(I-\sigma)^{T}Q(I-\sigma) Q^{-\frac{1}{2}}\right)Q^{\frac{1}{2}}y(t_k)\notag\\
    &\leq \eta_1y^{T}(t_k)Qy(t_k)\notag\\
    &=\eta_1 V(t_k).
\end{align*}
Thus, by applying Lemma \ref{lemma2.5}, we can deduce that the impulsive system (\ref{eq2.4}) is globally Mittag-Leffler stable.
\end{proof}

\begin{theorem}\label{th4.2}
Assume that the impulsive operator $U_k(x(t_k))$ satisfies
$$U_k(x(t_k))=-\sigma(x(t_k)-x^*),k=1,2,\cdots,m.$$
If there exist symmetric and positive definite matrix $Q > 0$, and positive scalars $ \rho_2>0, \mu_2>0$ and $0<\eta_2\leq1$, such that the matrix $Q > 0$ and the impulse matrix $\sigma$ satisfy the following inequalities
\begin{equation}\label{eq4.5}
  -2Q+\rho_2^{-1}Q^2+\rho_2(I-\rho A)^{T}(I-\rho A)+\mu_2Q \leq 0,
   \end{equation}
and
\begin{equation}\label{eq4.6}
   (I-\sigma)^{T}Q(I-\sigma)-\eta_2Q\leq 0,
\end{equation}
then the neural networks (\ref{eq2.4}) is globally Mittag-Leffler stable.
\end{theorem}
\begin{proof}
Consider a Lyapunov function candidate:
$$V(t)=y^{T}(t)Qy(t).$$

When $t\neq t_k$, by calculating the $\alpha$-order Caputo derivatives of $V(t)$ along the trajectories of system (\ref{eq4.3}), it follows
from Lemmas \ref{lemma2.1}, Lemmas \ref{lemma2.2} and Lemmas \ref{lemma2.3} and (\ref{eq4.5}) that
\begin{align}\label{eq4.4}
\leftidx{_{ 0}^C}D{_t^{\alpha}}V(t)&=\leftidx{_{ 0}^C}D{_t^{\alpha}}\big(y^{T}(t)Qy(t)\big)\notag\\
                                   &\leq 2y^{T}(t)Q \left(\leftidx{_{ 0}^C}D{_t^{\alpha}}y(t)\right)\notag\\
                                   &=2y^{T}(t)Q\left(-y(t)+\tilde{P}_{K}(y(t))\right)\notag\\
                                   &=y^{T}(t)(-2Q)y(t)+2y^{T}(t)Q\tilde{P}_{K}(y(t))\notag\\
                                   &\leq y^{T}(t)\left(-2Q+\rho_2^{-1}Q^2+\rho_2(I-\rho A)^{T}(I-\rho A)\right)y(t)\notag\\
                                   &=y^{T}(t)\left(-2Q+\rho_2^{-1}Q^2+\rho_2(I-\rho A)^{T}(I-\rho A)+\mu_2Q\right)y(t)-\mu_2y^{T}(t)Qy(t)\notag\\
                                   &\leq-\mu_2y^{T}(t)Qy(t)=-\mu_2V(t).
\end{align}
When $t=t_k$, it follows from ( \ref{eq4.6} ) that
\begin{align*}
V(t_k^{+})&=y^{T}(t_k^{+})Qy(t_k^{+})\notag\\
    &=\left(y(t_k^{-})-\sigma y(t_k^{-})\right)^{T}Q\left(y(t_k^{-})-\sigma y(t_k^{-})\right)\notag\\
    &=y^{T}(t_k)(I-\sigma)^{T}Q(I-\sigma)y(t_k)\notag\\
    &=y^{T}(t_k)\left((I-\sigma)^{T}Q(I-\sigma)-\eta_2Q\right)y(t_k)+\eta_2y^{T}(t_k)Qy(t_k)\notag\\
    &\leq \eta_2y^{T}(t_k)Qy(t_k)\notag\\
    &=\eta_2 V(t_k).
\end{align*}
Thus, by applying Lemma \ref{lemma2.5}, we can deduce that the impulsive system (\ref{eq2.4}) is globally Mittag-Leffler stable.
\end{proof}

\begin{remark}\label{rem4.1}
In \cite{WZ,WLH}, authors studied stability of fractional-order projection neural networks without impulse. Unlike the previous results, we have considered the stability of projection neural networks with impulsive effects by utilizing the general quadratic Lyapunov function. Moreover, the stability conditions are established in terms of different kinds of LMIs, which can be applied to check the stability of impulsive fractional-order  projection neural networks with different characteristics.
\end{remark}

\section{Numerical examples}\noindent
\setcounter{equation}{0}
In this section, two numerical examples for fractional-order  projection neural networks with impulses are given to illustrate the validity and feasibility of our main results.

\begin{example}
 Consider a two dimensional FPNNI as follows:
\begin{equation}\label{eq5.1}
\left\{
\begin{array}{l}
\leftidx{_0^C}D{_t^\alpha}x(t)=-x(t)+P_{K}\left(x(t) -\rho Ax(t)-\rho b\right),\ t\in J', \\
\triangle x(t_k) =x(t_k^{+})-x(t_k^{-})=U_{k}(x(t_k))=-\sigma(x(t_k)-x^*),\\
x(0)=x_0,k=1,2,\cdots,m,
\end{array}
\right.
\end{equation}
where
$$K=\{x|-2\leq x_{i} \leq 2,i=1,2 \}, \alpha=0.9,$$
and
$$A=\begin{pmatrix}
7 & -3  \\
-4 & 2
\end{pmatrix},
\quad b=\begin{pmatrix}
1 \\
-1
\end{pmatrix}, \quad \sigma=\begin{pmatrix}
0.5 & 0  \\
0 & 0.25
\end{pmatrix}.
$$
Let $\rho=0.1$, $T=2.5$, $\rho_1=1$, $\eta_1=1$ and $Q=\text diag (1,1).$  Then we obtain
 $$\lambda_{min}\left(-Q^{-\frac{1}{2}}\left(-2Q+\rho_1^{-1}Q^2+\rho_1(I-\rho A)^{T}(I-\rho A)\right) Q^{-\frac{1}{2}}\right)=\lambda_{min}\begin{pmatrix}
0.75 & -0.41  \\
-0.41 & 0.27
\end{pmatrix}=0.0349>0$$
and
 $$Q^{-\frac{1}{2}}(I-\sigma)^{T}Q(I-\sigma) Q^{-\frac{1}{2}}-\eta_1 I=\begin{pmatrix}
-0.75 & 0  \\
0 & -0.4375
\end{pmatrix}\leq 0,$$
which yields that conditions (\ref{eq4.1}) and  (\ref{eq4.2}) of Theorem \ref{th4.1} are satisfied. Therefore, it follows from Theorem \ref{th4.1} that system (\ref{eq5.1}) is globally Mittag-Leffler stable. By the iterative method to find the equilibrium point employed in \cite{ZS}, we can check that the equilibrium point of system (\ref{eq5.1}) is $(0.5,1.5)^\top$. Figure \ref{ex-1} shows the time responses of the variables $x_1(t)$ and $x_2(t)$ of system (\ref{eq5.1}).
\end{example}
\begin{figure}[H]
  \centering
  \includegraphics[width=5.2in]{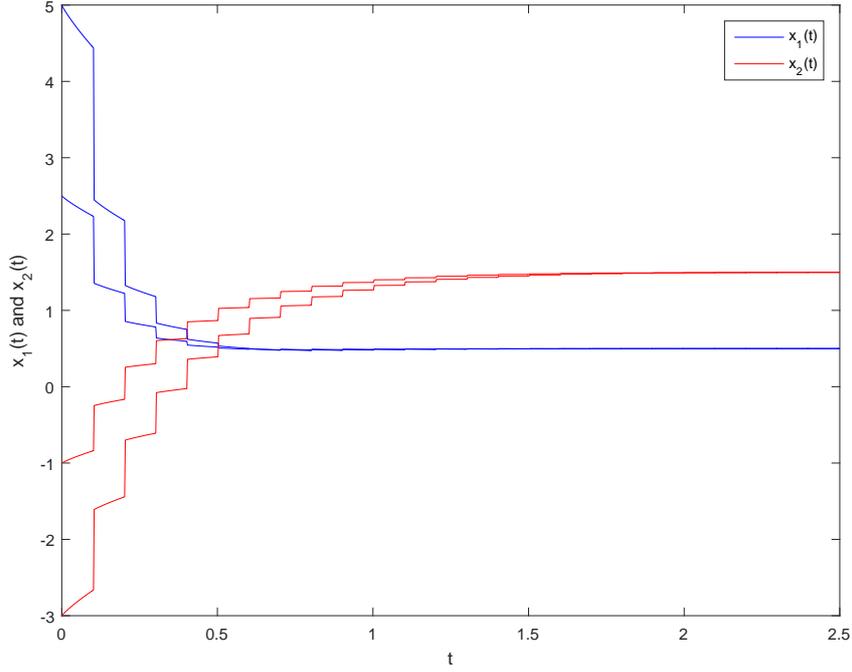}\\
  \caption{Transient states of the solutions $(x_1(t),x_2(t))^\top$ for system (\ref{eq5.1}) with the initial values $x_0=(5.0,-3.0)^\top$, $(2.5,-1.0)^\top$, respectively.}\label{ex-1}
\end{figure}

\noindent {\bf Example 5.2} Consider a two dimensional FPNNI as follows:
\begin{equation}\label{eq5.2}
\left\{
\begin{array}{l}
\leftidx{_0^C}D{_t^\alpha}x(t)=-x(t)+P_{K}\left(x(t) -\rho Ax(t)-\rho b\right),\ t\in J', \\
\triangle x(t_k) =x(t_k^{+})-x(t_k^{-})=U_{k}(x(t_k))=-\sigma(x(t_k)-x^*),\\
x(0)=x_0, \; k=1,2,\cdots,m,
\end{array}
\right.
\end{equation}
where
$$K=\{x|-1\leq x_{i} \leq 1,i=1,2 \},\alpha=0.7,$$
 and
$$A=\begin{pmatrix}
6 & -2  \\
-4 & 3
\end{pmatrix},
\quad b=\begin{pmatrix}
1 \\
0
\end{pmatrix}, \quad \sigma=\begin{pmatrix}
0.3 & 0  \\
0 & 0.3
\end{pmatrix}.
$$
Let $\rho=0.1$, $T=1.5$, $\rho_2=1$ and $\mu_2=0.1$. Then, by employing the Matlab LMI Toolbox, we can verify that LMI (\ref{eq4.5}) in Theorem \ref{th4.2} is feasible and the feasible solution is given as follows:
$$Q=\begin{pmatrix}
0.5911 & 0.1035 \\
0.1035 & 0.6515
\end{pmatrix}.$$
Clearly, the condition (\ref{eq4.6}) is satisfied with $\eta_2=1$. Thus, it follows from Theorem \ref{th4.2} that system (\ref{eq5.2}) is globally Mittag-Leffler stable. By the iterative method to find the equilibrium point used in \cite{ZS}, we can obtain that the equilibrium point of system (\ref{eq5.2}) is $(-0.3,-0.4)^\top$. Figure \ref{ex-2} shows the time responses of the variables $x_1(t)$ and $x_2(t)$ of system (\ref{eq5.2}).
\begin{figure}[H]
  \centering
  \includegraphics[width=5.2in]{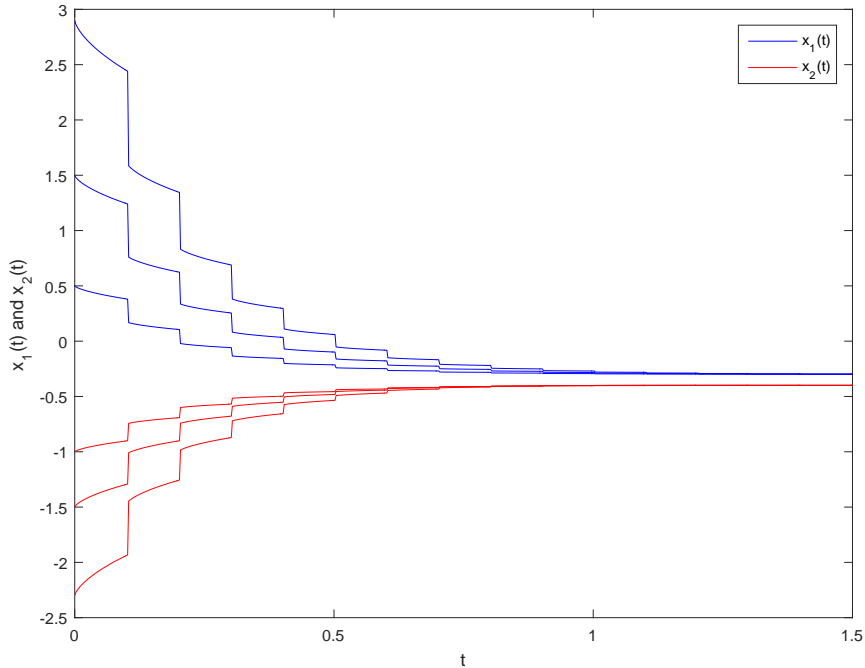}\\
  \caption{Transient states of the solutions $(x_1(t),x_2(t))^\top$ for system (\ref{eq5.2}) with the initial values $x_0=(2.9,-2.3)^\top$, $(1.5,-1.5)^\top$, $(0.5,-1.0)^\top$, respectively.}\label{ex-2}
\end{figure}

\section{Conclusions}\noindent
\setcounter{equation}{0}

This paper focused on the study of a new class of FPNNI under some mild conditions.  The main contributions of this paper are as follows: (i) a new class of FPNNI are introduced which capture the desired features of both the variational inequality and the fractional-order impulsive dynamical systems within the same framework;  (ii) some properties of the solutions set of FPNNI are obtained under different conditions; (iii) some sufficient conditions are given for ensuring the global Mittag-Leffler stability of the equilibrium point of FPNNI by utilizing a general quadratic Lyapunov function.

It is well known that the presence of delay often leads to the instability and oscillation of neural networks. Therefore, it is significant to study the stability of neural networks with delays. As far as we know, there are few papers on studying the stability of the fractional-order projection neural networks with both impulses and delays. Thus, it would be interesting and important to consider the fractional-order impulsive projective neural networks with delays. We leave this for our future research.

\end{document}